\documentclass[11pt, oneside]{article}   	
\usepackage{graphicx}				
\usepackage{amssymb}
\usepackage{amsthm}
\usepackage{amsmath}

\newcommand{\nsphere}{\mathbb{S}^{n-1}}
\newcommand\nspace{\mathbb{R}^n}
\newcommand\innerproduct[2]{\langle {#1},{#2} \rangle}

\newcommand{\at}{\makeatletter @\makeatother}
\newtheorem{theorem}{Theorem}
\newtheorem{lemma}[theorem]{Lemma}

\title{A note on spherical discrepancy}
\author{Yossi Lonke}
\date{}							

\begin{document}
\begin{titlepage}
  \maketitle
  \begin{abstract}
  A non-algorithmic, generalized version of a recent result, asserting that a natural relaxation of the Koml\'os conjecture
 from boolean discrepancy to spherical discrepancy is true, is proved by a very short argument using convex geometry.
  \end{abstract}
\vfill
Yossi Lonke\\
yossilonke\at me.com \\
Independent Research \\
13 Basel Street, Tel-Aviv ,Israel\\
Orcid Id: 0000-0001-7493-5085\\
\end{titlepage}

\section*{}

The term  \emph{spherical discrepancy}  has been recently introduced (\cite{JM}) in the context of an optimization problem on the unit sphere
$\mathbb{S}^{n-1}=\{x\in\mathbb{R}^n: \|x\|_2=1\}$. Given a collection of unit vectors $u_1,\dots,u_m$ in $\nsphere$  the problem is to minimize 
\begin{equation}\label{eq:minProb}
 \min_{x\in\nsphere}\max_i\innerproduct{u_i}{x}. 
 \end{equation}
Spherical discrepancy is a relaxation of boolean discrepancy, whereby the minimization in (\ref{eq:minProb}) is to be carried over the set of sign-vectors $\{-1,1\}^n$, instead of over the sphere.
A famous conjecture made by Koml\'os,\footnote{There appears to be no source written by Koml\'os himself containing this conjecture.} is that whenever
the vectors $u_i$  have Euclidean norm at most $1$, one can find a sign-vector $x\in\{-1,1\}^n$ such that the maximum in (\ref{eq:minProb}) is $O(1)$. Arranging the vectors $u_i$ as $n$ columns of an $m\times n$ matrix, say $W$, the problem is to minimize $\|Wx\|_\infty$ over the sign-vectors $x\in \{-1,1\}^n$. It can be assumed that $m=n$ (\cite{S}, Section~4). The Koml\'os conjecture has already resisted the efforts of two generations of mathematicians, hence there is a growing body of "relaxed Komlo\'s conjectures", which replace the set $\{-1,1\}^n$ by larger sets, or by other objects. See \cite{CGRT}, \cite{N}.
One such relaxed Koml\'os problem is discussed in \cite{JM}. Section~4 thereof contains the following result. 
\bigskip

\begin{theorem} Let $w_1,\dots, w_n$ be vectors with $\|w_i\|_2\leq 1$, and let $W$ be the matrix with the $w_i$ as columns. Then we can find a unit vector $x\in\nsphere$ such that
\begin{equation}
\|Wx\|_\infty = O(\frac{1}{\sqrt{n}})
\end{equation}
in time polynomial in $n$.
\end{theorem}

The proof of Theorem 1 in  \cite{JM} is constructive, i.e. it presents an algorithm, running in polynomial time, that actually finds a minimizer. The purpose of this note is to demonstrate that if one is merely interested in the \emph{existence} of a minimizer, then a much simpler argument, (and slightly more general), is possible. It is based on the following observation. 
\begin{lemma} Let $K,L$ be two unit balls of norms in $\nspace$, 
and let $W$ be an $n\times n$ nonsingular matrix. Then
\begin{equation}\label{eq:myIneq}
\min_{\|x\|_K=1}\|Wx\|_L \leq \left(\frac{\hbox{\rm Vol}\,(W(B_K))}{\hbox{\rm Vol}\,(B_L)}\right)^{1/n}
\end{equation}
\end{lemma}
\begin{proof}
The minimum $\alpha=\min_{\|x\|_K=1}\|Wx\|_L$ is attained by compactness, and is positive because $W$ is nonsingular. If ${\alpha y\notin W(K)}$ for some $y\in L$, then
for some $0<|t|<1$ the vector $t\alpha y$ would belong to the boundary of $W(K)$, hence the vector $z=W^{-1}(t\alpha y)$ would belong to the boundary of~$K$, but then $\|Wz\|_L<\alpha$, contradicting
 the minimality of~$\alpha$. Therefore,  we must have $\alpha L\subset W(K)$. Taking the volumes of the two sides of this inclusion yields inequality (\ref{eq:myIneq}). 
\end{proof}
A nonconstructive version of a slightly more general theorem than Theorem 1 follows easily from Lemma 2.
\begin{theorem} Let $w_1,\dots, w_n$ be vectors whose mean Euclidean norm is at most $1$, and let $W$ be the matrix with the $w_i$ as columns. Then there exists a unit vector $x\in\nsphere$ such that
\begin{equation}
\|Wx\|_\infty = O(\frac{1}{\sqrt{n}})
\end{equation}
\end{theorem} 
\begin{proof}
There is no loss of generality in assuming that $W$ is nonsingular, because we can 
perturb the $w_i$'s so that they are in general position while keeping their norm.
Take $K=B_2^n$, the Euclidean unit ball, and $L=B_\infty^n$, the unit ball of the $\|\cdot\|_\infty$ norm, i.e., the cube $[-1,1]^n$. We have 
\[\textrm{Vol}\,B_2^n=\frac{\pi^{n/2}}{\Gamma(\frac{n}{2}+1)}\quad\textrm{and}\quad \textrm{Vol}\,B_\infty^n=2^n. \]
By Lemma~2 and Stirling's formula,
\begin{equation}\label{eq:estimate}
\min_{\|x\|_2=1}\|Wx\|_\infty\leq (|\det W|)^{1/n}\frac{\sqrt{\pi}}{2\Gamma(\frac{n}{2}+1)^{1/n}}\sim (|\det W|)^{1/n}\frac{\sqrt{2\pi e}}{2\sqrt{n}}
\end{equation}
(where the symbol $\sim$ signifies that the two sides are asmyptotically equivalent as $n\to\infty$.)
By Hadamard's inequality for the determinant,
\[ (|\det{W}|)^{1/n}\leq \left(\prod_{i=1}^n\|w_i\|_2\right)^{1/n} \leq\frac{1}{n}\sum_{i=1}^n\|w_i\|_2\]
Hence, by (\ref{eq:estimate}) and the hypothesis, 
\[\min_{\|x\|_2=1}\|Wx\|_\infty=O(\frac{1}{\sqrt{n}}).\]
\end{proof}

\end{document}